\documentclass[a4paper,reqno,11pt]{amsart}

\usepackage{amssymb}
\usepackage{amstext}
\usepackage{amsmath}
\usepackage{amscd}
\usepackage{amsthm}
\usepackage{amsfonts}
\usepackage{enumerate}
\usepackage{graphicx}
\usepackage{latexsym}
\usepackage[all]{xy}
\input xy
\xyoption{all}
\usepackage[usenames]{color}

\newtheorem{theorem}{Theorem}[section]

\newtheorem{corollary}[theorem]{Corollary}
\newtheorem{lemma}[theorem]{Lemma}
\newtheorem{proposition}[theorem]{Proposition}
\newtheorem{question}[theorem]{Question}

\theoremstyle{definition}
\newtheorem{definition}[theorem]{Definition}
\newtheorem{remark}[theorem]{Remark}
\newtheorem{example}[theorem]{Example}
\newtheorem*{ac}{{\sc Acknowledgments}}
\newtheorem{definitiontheorem}[theorem]{Definition-Theorem}

\newcommand{\TT}{{\mathcal T}}
\newcommand{\MM}{{\mathcal M}}

\newcommand{\add}{\mathsf{add}}

\newcommand{\thick}{\mathsf{thick}}
\newcommand{\Hom}{\operatorname{Hom}\nolimits}
\newcommand{\silt}{\operatorname{silt}\nolimits}
\newcommand{\tilt}{\operatorname{tilt}\nolimits}

\renewcommand{\mod}{\operatorname{mod}} 
\newcommand{\proj}{\operatorname{proj}} 

\newcommand{\homo}[3]{{\rm{Hom}}_{#1}(#2, #3)} 
\newcommand{\ext}[4]{{\rm{Ext}}_{#1}^{#2}(#3, #4)} 
\newcommand{\ehomo}[2]{{\rm{End}}_{#1}(#2)} 

\newcommand{\length}[1]{\mathsf{length}(#1)} 

\newcommand{\CCC}{\mathcal{C}} 
\newcommand{\EEE}{\mathcal{E}} 

\newcommand{\ann}[1]{\mathsf{ann}#1} 
\newcommand{\torsion}{\operatorname{\mathsf{tor}}}

\newtheorem{assumption}[theorem]{Assumption}

\newcommand{\K}{\mathsf{K}}

\renewcommand{\mod}{\operatorname{mod}\nolimits}

\numberwithin{equation}{theorem}

\begin{document}

\title{Tilting-connected symmetric algebras}
\author{Takuma Aihara}
\address{Division of Mathematical Science and Physics, Graduate School of Science and Technology, Chiba University, 
Yayoi-cho, Chiba 263-8522, Japan}
\email{taihara@math.s.chiba-u.ac.jp}
\begin{abstract}
The notion of silting mutation was introduced by Iyama and the author. 
In this paper we mainly study silting mutation for self-injective algebras 
and prove that any representation-finite symmetric algebra is tilting-connected. 
Moreover we give some sufficient conditions for a Bongartz-type Lemma to hold for silting objects. 
\end{abstract}
\maketitle


\section{Introduction}\label{intro}

The study of derived categories is important in many branches of mathematics, 
for example representation theory and algebraic geometry. 
Derived categories include much of the information of an abelian category. 
The study of equivalences of derived categories is a major subject. 
In representation theory, it is important that derived equivalences preserve many homological properties, 
e.g. an algebra being symmetric, self-injective or Iwanaga-Gorenstein. 
Moreover representation-finiteness of self-injective algebras is also derived invariant. 

From a Morita theoretic viewpoint \cite{Ri1}, tilting complexes are crucial. 
Moreover tilting mutation of a tilting complex often plays an important role 
for Brou\'e's abelian defect group conjecture and Brauer tree algebras. 
In the study of Brou\'e's conjecture, 
Okuyama's method is often employed to show that the conjecture holds in several cases  \cite{O, Ri2}.
This is nothing but taking iterated tilting mutation of tilting complexes. 
On the other hand, the author introduced mutation of Brauer trees 
and proved that it is compatible with tilting mutation for Brauer tree algebras \cite{A}. 

In this paper we mainly consider symmetric algebras. 
Then tilting mutation acts on the set of basic tilting complexes. 
We say that a symmetric algebra is \emph{tilting-connected} 
if the action of iterated irreducible tilting mutation on the set of basic tilting complexes is transitive: 
see Definition \ref{silting connected}. 
We pose the following natural question: see Question \ref{silting-connected question}. 

\begin{question}\label{tilting-connected question}
When is a symmetric algebra tilting-connected?
\end{question}

It was shown in \cite{AI} that a local symmetric algebra is tilting-connected: see Theorem \ref{local and hereditary}. 
On the other hand, there is a symmetric algebra which is not tilting-connected \cite{AGI}. 

A main result of this paper is the following. 

\begin{theorem}\label{symm finite case1}
Any representation-finite symmetric algebra is tilting-connected. 
\end{theorem}

As another theme of this paper, we consider a Bongartz-type Lemma. 
We recall the definition of a (classical) tilting module. 
Let $A$ be a finite dimensional algebra over a field. 
An $A$-module $T$ is called a \emph{tilting module} if it satisfies
(i) $\ext{A}{1}{T}{T}=0$, 
(ii) the projective dimension of $T$ is at most one, and
(iii) there exists an exact sequence $0\to A\to T_0\to T_1\to0$ with $T_0,T_1$ in $\add T$.

Bongartz proved the result \cite{B}. 

\begin{theorem}[Bongartz]\label{Bongartz}
Any $A$-module satisfying (i) and (ii) above is a direct summand of a tilting module. 
\end{theorem}

It is natural to consider a Bongartz-type Lemma for complexes over a finite dimensional algebra $A$. 
We say that a complex $T$ in $\K^{\rm b}(\proj A)$ is \emph{pretilting} 
if it satisfies $\homo{\K^{\rm b}(\proj A)}{T}{T[i]}=0$ for any non-negative integer $i$, 
and \emph{partial tilting} if it is a direct summand of a tilting complex. 
Cleary any partial tilting complex is pretilting, 
but the converse is not true in general (e.g. \cite{Ri1}). 
It is natural to ask when the converse is true. 
We pose the following question: see Question \ref{Bongartz-type question}. 

\begin{question}\label{Bongartz-type question1}
Let $A$ be a symmetric algebra. 
Is any pretilting complex partial tilting?
\end{question}

Abe and Hoshino proved that Question \ref{Bongartz-type question1} has a positive answer 
if $A$ is a representation-finite symmetric algebra \cite{AH}. 
We give a simple proof of their result by applying our theory of tilting mutation: see Corollary \ref{symm finite Bongartz-type}. 

This paper is organized as follows:

In Section \ref{silting}, we study silting mutation of silting objects and a partial order on the set of basic silting objects, 
which have been introduced by Iyama and the author \cite{AI}. 
We give two different Bongartz-type Lemmas (Proposition \ref{Bongartz-type} and Theorem \ref{finite mutation1}). 
They play an important role in later sections. 

In Section \ref{silting transitivity}, we consider silting transitivity of silting objects. 
We give a sufficient condition for two silting objects 
to be iterated irreducible silting mutation of each other (Theorem \ref{finite mutation2}). 
This is the first main step in the proof of our main theorem.

In Section \ref{torsion}, we show that any covariantly finite torsion class gives rise to a silting object 
(Theorem \ref{silting induced by torsion pair}). 
This is the second main step in the proof of our main theorem.
Our construction is generalization of that of Okuyama-Rickard complexes (cf. \cite{AI, O}): see Example \ref{Okuyama-Rickard}. 

In Section \ref{symmetric}, we prove our main theorem: 
that any representation-finite symmetric algebra is tilting-connected (Theorem \ref{symm finite case}). 
The important result (Proposition \ref{A}) seems to be very interesting by itself. 
Moreover we give some examples of the behavior of silting mutation. 

\medskip
{\bf Notation}
Let $\TT$ be an additive category.
For an object $X$ of $\TT$,
we denote by $\add X$ the smallest full subcategory of $\TT$ containing $X$
which is closed under finite direct sums, summands and isomorphisms. 

Let $A$ be a finite dimensional algebra over a field. 
We denote by $\mod A$ ($\proj A$) the category of finitely generated (projective) right $A$-modules and 
by $\K^{\rm b}(\mod A)$ and $\K^{\rm b}(\proj A)$ the bounded homotopy category of $\mod A$ and $\proj A$, respectively. 
For an integer $i$, we denote by $H^i:\K^{\rm b}(\mod A)\to\mod A$ the $i$-th cohomological functor. 
An $A$-module means a finitely generated right $A$-module. 
We always assume that an algebra is basic. 

\begin{ac}
The author would like to express his deep gratitude to Osamu Iyama and Joseph Grant
who read the paper carefully and gave a lot of helpful comments and suggestions.
\end{ac}


\section{Silting objects and silting mutation}\label{silting}

\subsection{Preliminaries}

The aim of this subsection is to study silting mutation and a partial order in the sense of \cite{AI} 
and to give some results which are necessary for the rest of this paper. 

Throughout this paper let $\TT$ be a triangulated category and we assume the following:
\begin{assumption}\label{assumption:triangulated cat}
$\TT$ is Krull-Schmidt, 
$k$-linear for a field $k$ and 
\emph{Hom-finite}, that is, ${\rm dim}_{k}\homo{\TT}{X}{Y}<\infty$ for all objects $X, Y$ in $\TT$. 
\end{assumption}

Let $X$ be an object of $\TT$. 
We denote by $\thick X$ the smallest thick subcategory of $\TT$ containing $X$. 
We say that $X$ is \emph{basic} if the endomorphism algebra of $X$ is basic. 

Let us start with the definition of silting objects.

\begin{definition}
Let $T$ be an object of $\TT$. 
\begin{enumerate}[(a)]
\item We say that $T$ is \emph{presilting} 
if $\homo{\TT}{T}{T[i]}=0$ for any positive integer $i>0$. 
\item We say that $T$ is \emph{silting} 
if it is presilting 
and satisfies $\TT=\thick{T}$. 
We denote by $\silt{\TT}$ 
the isomorphism classes of basic silting 
objects in $\TT$. 
\item We say that $T$ is \emph{partial silting} 
if it is a direct summand of a silting 
object. 
\end{enumerate}
\end{definition}

We say that a morphism $f:X\to Y$ is \emph{left minimal} 
if any morphism $g:Y\to Y$ satisfying $gf=f$ is an isomorphism. 
Dually we define a \emph{right minimal} morphism.

Let $\MM$ be a full subcategory of $\TT$.
We say that a morphism $f:X\to M$ is a \emph{left $\MM$-approximation} of $X$ 
if $M$ belongs to $\MM$ and $\homo{\TT}{f}{M'}$ is surjective for any object $M'$ in $\MM$. 
We say that $\MM$ is \emph{covariantly finite} if any object in $\TT$ has a left $\MM$-approximation. 
Dually, we define a \emph{right $\MM$-approximation} and a \emph{contravariantly finite subcategory}. 
We say that $\MM$ is \emph{functorially finite} if it is covariantly and contravariantly finite. 

Assumption \ref{assumption:triangulated cat} implies that 
$\add M$ is a functorially finite subcategory for any object $M$ in $\TT$. 
To see this, let $X$ be an object of $\TT$ and $\{f_1,\cdots,f_n\}$ a $k$-basis of the $k$-vector space $\homo{\TT}{X}{M}$. 
Then the morphism 
\[\def\arraystretch{0.8}
\left[\begin{array}{c}
f_1\\f_2\\\vdots\\f_n\end{array}
\right]
:X\to M^{\oplus n}
\] 
is a left $\add M$-approximation of $X$, 
which implies that $\add M$ is covariantly finite in $\TT$. 
By a similar argument, we see that $\add M$ is contravariantly finite in $\TT$. 

Now we define silting mutation. 

\begin{definitiontheorem}\label{def:silting mutation}\cite{AI}
Let $T$ be a basic silting object in $\TT$. 
For a composition $T=X\oplus M$, we take a triangle 
\[\xymatrix{
X \ar[r]^{f} & M' \ar[r] & Y \ar[r] & X[1]
}\]
with a minimal left $\add{M}$-approximation $f:X\to M'$ of $X$. 
Then $\mu_{X}^{+}(T):=Y\oplus M$ is again a basic silting object,
and we call it a \emph{left mutation} of $T$ with respect to $X$. 
\end{definitiontheorem}

Dually we define a \emph{right mutation} $\mu_{X}^{-}(T)$ of $T$ with respect to $X$. 
We will say that \emph{(Silting) mutation} to mean either left or right mutation. 
We say that mutation is \emph{irreducible} if $X$ is indecomposable.


The following definition is very useful. 

\begin{definition}
%
For any objects $M, N$ of $\TT$, we write $M\geq N$ if $\homo{\TT}{M}{N[i]}=0$ for any positive integer $i>0$. 
\end{definition}


\begin{remark}
Note that the relation $\ge$ on objects of $\TT$ is far
from being a partial order. Nevertheless we use the notation $\ge$
since it is very simple and moreover $\geq$ is a partial order
on the set $\silt\TT$ by the following result.
\end{remark}

\begin{theorem}\label{partial order}\cite[Theorem 2.11, Proposition 2.14, Theorem 2.35]{AI}
The following hold:
\begin{enumerate}[{\rm (1)}]
\item $\geq$ gives a partial order on $\silt\TT$.
\item Let $T, U$ be basic silting objects in $\TT$. 
\begin{enumerate}[{\rm (i)}]
\item $T\geq U$ if and only if any object $X$ of $\TT$ with $U\geq X$ satisfies $T\geq X$.
\item If $U$ belongs to $\add T$, then $U$ is isomorphic to $T$.
\item The following are equivalent:
\begin{enumerate}[{\rm (a)}]
\item $U$ is an irreducible left mutation of $T$;
\item $T>U$ and there is no basic silting object $V$ in $\TT$ satisfying $T>V>U$. 
\end{enumerate}
\end{enumerate}
\end{enumerate}
\end{theorem}

We observe the relation $\geq$ on objects of $\K^{\rm b}(\proj A)$. 

For a finite dimensional algebra $A$ over a field $k$, 
we denote by $D:=\Hom_k(-,k):\mod A\leftrightarrow\mod A^{\rm op}$ the $k$-duality, 
where $A^{\rm op}$ is the opposite algebra of $A$. 
We denote by $\nu:=D\homo{A}{-}{A}:\mod A\to\mod A$ the Nakayama functor of $A$ 
and put $\nu^{-1}:=\homo{A}{DA}{-}$. 

The following result is useful to calculate $k$-vector spaces $\homo{\K^{\rm b}(\proj A)}{P}{Q}$ on $\K^{\rm b}(\proj A)$. 

\begin{lemma}\cite{H,HK}\label{Serre}
Let $A$ be a finite dimensional algebra over a field. 
Let $X$ be an object of $\K^{\rm b}(\mod A)$ and $P$ an object of $\K^{\rm b}(\proj A)$. 
Then we have a bifunctorial isomorphism 
\[\homo{\K^{\rm b}(\mod A)}{P}{X}\simeq D\homo{\K^{\rm b}(\mod A)}{X}{\nu P}.\]
\end{lemma}

\begin{example}\label{symmetric silting=tilting}\cite[Example 2.8]{AI}
If $A$ is a symmetric algebra, 
then any (pre)silting object is (pre)tilting. 
Indeed  by Lemma \ref{Serre}, 
the condition that $A$ is symmetric implies that there exists an isomorphism $\homo{\TT}{T}{T[i]}\simeq D\homo{\TT}{T}{T[-i]}$ 
for all integers $i$. 
\end{example}

Let $X$ be a non-zero complex in $\K^{\rm b}(\proj A)$ for a finite dimensional algebra $A$. 
We define the \emph{length} of $X$ as $\length{X}=b-a+1$ whenever $X^i=0$ for $i<a\leq b<i$ and $X^a\neq0, X^b\neq0$. 

The observation below shows that $\geq$ is closely related to the length of a complex. 

\begin{proposition}\label{prop:partial order and length}
Let $A$ be a finite dimensional algebra over a field and $\TT:=\K^{\rm b}(\proj A)$. 
Let $X$ be an object of $\TT$ and $\ell>0$. 
Then the following are equivalent:
\begin{enumerate}[{\rm (1)}]
\item $\length{X}\leq \ell$;
\item $A[n]\geq X\geq A[\ell+n-1]$ for some integer $n$.
\end{enumerate}
\end{proposition}
\begin{proof}
We show the implication (1)$\Rightarrow$(2). 
Since $\length{X}\leq\ell$, 
we can write 
\[X=(\cdots\to0\to X^{-\ell-n+1}\to\cdots\to X^{-n}\to0\to\cdots)\] 
for some $n\in\mathbb{Z}$. 
This immediately implies $A[n]\geq X\geq A[\ell+n-1]$. 

We show the implication (2)$\Rightarrow$(1). 
Applying shifts, we can assume $n=0$. 
As $A\geq X$, we observe $H^i(X)=0$ for any $i>0$. 
Since any term of $X$ is a projective module, 
we see that $X$ is isomorphic to a complex $(\cdots\to Y^{-1}\to Y^0\to0\to\cdots)$. 
By Lemma \ref{Serre}, we have an isomorphism $\homo{\K^{\rm b}(\mod A)}{A}{\nu X[i]}\simeq D\homo{\TT}{X[i]}{A}$.
As $X\geq A[\ell-1]$, we see $H^i(\nu X)=0$ for any $i\leq-\ell$. 
Since any term of $\nu X$ is an injective module, 
we obtain that $\nu X$ is isomorphic to a complex $(\cdots\to0\to Z^{-\ell+1}\to Z^{-\ell+2}\to\cdots)$. 
This implies $X\simeq (\cdots\to0\to \nu^{-1}Z^{-\ell+1}\to \nu^{-1}Z^{-\ell+2}\to\cdots)$. 
Thus $X$ can be represented by a complex 
\[(\cdots\to0\to X^{-\ell+1}\to\cdots\to X^0\to0\to\cdots).\]
Hence the assertion holds. 
\end{proof}

The following results play an important role later. 

\begin{proposition}\cite[Proposition 2.23]{AI}\label{resolution}
Let $T$ be in $\silt\TT$. For any object $U$ of $\TT$ with $T\geq U=U_0$, we have triangles
\[\xymatrix@R=0.2cm{
U_1\ar[r]^{g_1}&T_0\ar[r]^{f_0}&U_0\ar[r]&U_1[1],\\
\cdots,\\
U_\ell\ar[r]^{g_\ell}&T_{\ell-1}\ar[r]^{f_{\ell-1}}&U_{\ell-1}\ar[r]&U_\ell[1],\\
0\ar[r]^{g_{\ell+1}}&T_\ell\ar[r]^{f_\ell}&U_\ell\ar[r]&0,
}\]
for some $\ell\ge0$ such that $f_i$ is a minimal right $\add{T}$-approximation and
$g_{i+1}$ belongs to $J_{\TT}$ for any $0\le i\le \ell$ 
where $J_{\TT}$ is the Jacobson radical of $\TT$.
\end{proposition}

\begin{lemma}\cite[Lemma 2.24]{AI}\label{no common summand}
Let $T$ be in $\silt\TT$ and $U_0,U'_0$ be objects of $\TT$ such that $T\geq U_0,U_0'$.
For $U_0$, we take $\ell\ge0$ and so obtain triangles as in Proposition \ref{resolution}.
Also for $U_0'$, we get triangles 
\[\xymatrix@R=0.2cm{
U_1'\ar[r]^{g_1'}&T_0'\ar[r]^{f_0'}&U_0'\ar[r]&U_1'[1],\\
&\cdots,\\
U_{\ell'}'\ar[r]^{g_\ell'}&T_{\ell'-1}'\ar[r]^{f_{\ell'-1}'}&U_{\ell'-1}'\ar[r]&U_{\ell'}'[1],\\
0\ar[r]^{g_{\ell'+1}'}&T_{\ell'}'\ar[r]^{f_{\ell'}'}&U_{\ell'}'\ar[r]&0,
}\]
satisfying the same properties.
If $\Hom_{\TT}(U_0,U_0'[\ell])=0$ holds, then we have $(\add T_\ell)\cap(\add T_0')=0$.
\end{lemma}

We improve the result \cite[Proposition 2.36]{AI}. 

\begin{proposition}\label{make closer}
Let $T$ be in $\silt{\TT}$ and let $U$ be a presilting object of $\TT$ which does not belong to $\add T$. 
If $T\geq U$, 
then there exists an irreducible left mutation $P$ of $T$ such that $T>P\geq U$.
\end{proposition}
\begin{proof}
This is a simple modification of the proof of \cite[Proposition 2.36]{AI}, 
but for the convenience of the reader we give full details here. 
Since $U\not\in\add{T}$ and $T\geq U$,
we obtain triangles as in Proposition \ref{resolution} with $\ell>0$. 
Now take an indecomposable object $X$ of $T_\ell$ and put $P:=\mu^{+}_{X}(T)$. 
By Theorem \ref{partial order}, we have $T>P$. 
To show $P\geq U$, 
we consider the triangle as in Definition-Theorem \ref{def:silting mutation}, 
i.e., $T=X\oplus M, P=Y\oplus M$ and the triangle $X\xrightarrow{f}M'\to Y\to X[1]$. 
Since we have an exact sequence 
\[\homo{\TT}{X}{U[i]}\to\homo{\TT}{Y}{U[i+1]}\to\homo{\TT}{M'}{U[i+1]},\]
we find that $\homo{\TT}{P}{U[i]}=0$ for any $i>1$. 
Thus it remains to prove $\homo{\TT}{P}{U[1]}=0$. 
Since we have an exact sequence 
\[\homo{\TT}{M'}{U}\stackrel{\cdot f}{\to}\homo{\TT}{X}{U}\to\homo{\TT}{Y}{U[1]}\to\homo{\TT}{M'}{U[1]}=0,\]
we only need to show that $\homo{\TT}{M'}{U}\stackrel{\cdot f}{\to}\homo{\TT}{X}{U}$ is surjective. 
Fix $a:X\to U$ and consider a diagram 
\[\xymatrix{
& Y[-1] \ar[r] & X \ar[r]^{f}\ar[d]^{a} & M' \\
U'_1 \ar[r]_{g'_1} & T'_0 \ar[r]_{f'_0} & U \ar[r] & U'_1[1]
}\]
where the lower triangle is given in Lemma \ref{no common summand} as $U'_0=U$. 
Since $f'_0$ is a right $\add{T}$-approximation, 
we get a morphism $b:X\to T'_0$ with $a=f'_0 b$. 
Since $\add{T_\ell}\cap\add{T'_0}=0$ by Lemma \ref{no common summand}, 
we have $X\not\in\add{T'_0}$ and so $T'_0\in\add{M}$. 
Since $f$ is a left $\add{M}$-approximation, 
we obtain $c:M'\to T'_0$ with $b=cf$. 
Thus we see that $a=(f'_0 c)f$ and the assertion holds. 
\end{proof}

\subsection{A Bongartz-type Lemma for silting objects}

In this subsection we consider a Bongartz-type Lemma in a triangulated category $\TT$. 

Let us start with the following observation. 

\begin{proposition}\cite[Corollary 2.28]{AI}\label{silting indecomposable number}
All silting objects in $\TT$ have the same number of non-isomorphic indecomposable direct summands.
\end{proposition}

In general a Bongartz-type Lemma does not hold for tilting objects in $\TT$, 
but we can hope that the question below has a positive answer.

\begin{question}\label{Bongartz-type question}
Is any presilting object in $\TT$ partial silting?
\end{question}

If $A$ is a symmetric algebra and $\TT:=\K^{\rm b}(\proj A)$, 
then Question \ref{Bongartz-type question1} is nothing but Question \ref{Bongartz-type question}: 
see Example \ref{symmetric silting=tilting}.

When $A$ is an algebra presented by a quiver $Q$ with relations, 
for any vertex $i$ of $Q$ we denote by $S_i$ and $P_i$ a simple $A$-module and an indecomposable projective $A$-module 
corresponding to $i$, respectively.

\begin{example}\label{Dynkin A3}
Let $A$ be the path algebra of the quiver $1\to 2\to 3$. 
Then for any integer $n$, $T:=P_3\oplus S_1[n]$ is a pretilting object in $\K^{\rm b}(\proj A)$ 
but not a partial tilting object  (see \cite[Section 8]{Ri1}).
On the other hand, we see that $T$ is a partial silting object. 
For each choise of $n\in\mathbb{Z}$, 
$\{M_\ell \ |\ \ell\in\mathbb{Z}\}$ gives a complete lists of complements to $T$, where
\begin{enumerate}[{\rm (1)}]
\item if $n\geq0$, then 
\[M_\ell=\begin{cases}
\def\arraystretch{0.5}
\ \left(\begin{array}{c}1\\2\end{array}\right)[\ell] & \mbox{if }\ell<0 \\
\def\arraystretch{0.5}
\ \left(\begin{array}{c}1\\2\\3\end{array}\right)[\ell] & \mbox{if }0\leq\ell\leq n \\
\def\arraystretch{0.5}
\ \left(\begin{array}{c}2\\3\end{array}\right)[\ell] & \mbox{if }\ell>0
\end{cases}\]
\item if $n<0$, then 
\[M_\ell=\begin{cases}
\def\arraystretch{0.5}
\ \left(\begin{array}{c}1\\2\end{array}\right)[\ell] & \mbox{if }\ell\leq n \\
\def\arraystretch{0.5}
\ \left(\begin{array}{c}2\end{array}\right)[\ell] & \mbox{if }n<\ell<0 \\
\def\arraystretch{0.5}
\ \left(\begin{array}{c}2\\3\end{array}\right)[\ell] & \mbox{if }\ell\geq0.
\end{cases}\]
\end{enumerate}
Here, we have denoted an $A$-module by its Loewy series. 
\end{example}

The following result is natural generalization of Bongartz's result \cite{B} (cf. \cite[Lemma 3.1]{AH}), 
which plays an important role in Section \ref{torsion}. 

\begin{proposition}\label{Bongartz-type}
Let $T$ be in $\silt\TT$. 
Then any presilting object $U$ in $\TT$ satisfying $T[-1]\geq U\geq T$ is partial silting. 
\end{proposition}
\begin{proof}
We take a triangle $V\to U'\xrightarrow{f} T\to V[1]$ with a right $\add{U}$-approximation $f:U'\to T$ of $T$.
Put $W:=U\oplus V$. 
Since $T$ is a silting object, we can check $\TT=\thick{W}$ easily. 

(i) We show $\homo{\TT}{U}{V[i]}=0$ for any $i>0$. 
Since we have an exact sequence 
\[\homo{\TT}{U}{T[i-1]}\xrightarrow{0} \homo{\TT}{U}{V[i]}\to \homo{\TT}{U}{U'[i]}=0,\]
we observe $\homo{\TT}{U}{V[i]}=0$ for any $i>0$. 

(ii) We show $\homo{\TT}{V}{W[i]}=0$ for any $i>0$. 
Since we have an exact sequence 
\[0\stackrel{\mbox{(i)}}{=}\homo{\TT}{U'}{W[i]}\to \homo{\TT}{V}{W[i]}\to \homo{\TT}{T[-1]}{W[i]}\to \homo{\TT}{U'[-1]}{W[i]}
\stackrel{\mbox{(i)}}{=}0,\]
we find an isomorphism $\homo{\TT}{V}{W[i]}\simeq \homo{\TT}{T[-1]}{W[i]}$. 
Since there is an exact sequence 
\[\homo{\TT}{T}{T[i]}\to \homo{\TT}{T}{V[i+1]}\to \homo{\TT}{T}{U'[i+1]}\stackrel{T\geq U[1]}{=}0,\]
we obtain $\homo{\TT}{T}{V[i+1]}=0$ for any $i>0$. 
This implies $\homo{\TT}{V}{W[i]}=0$ for any $i>0$. 

By (i) and (ii), we see $\homo{\TT}{W}{W[i]}=0$ for any $i>0$. 
Thus the assertion holds.
\end{proof}

The theorem below shows that 
finiteness of the number of basic silting objects up to shift implies a positive answer for Question \ref{Bongartz-type question}.

\begin{theorem}\label{finite mutation1}
Let $T$ be in $\silt\TT$ and $U$ be a presilting object of $\TT$ with $T\geq U$. 
If there exist only finitely many $V\in\silt\TT$ satisfying $T\geq V\geq U$, 
then $U$ is a partial silting object. 
\end{theorem}
\begin{proof}
We shall find a silting object $T'$ such that $U$ is isomorphic to a direct summand of $T'$. 
Assume that $U$ is not a partial silting object. 
By Proposition \ref{make closer}, 
we have an infinite sequence  
\[T=T_{0}>T_{1}>T_2>\cdots\]
of irreducible left mutations satisfying $T_i\geq U$ for any $i\geq0$. 
It is a contradiction since there exist only finitely many $V\in\silt\TT$ satisfying $T\geq V\geq U$. 
Hence $U$ is isomorphic to a direct summand of a silting object $T':=T_i$ for some $i\geq0$. 
\end{proof}


\section{Silting transitivity}\label{silting transitivity}

In this section we will discuss silting transitivity, i.e., the transitivity of irreducible silting mutation on $\silt\TT$. 

The following definition is very useful. 

\begin{definition}\label{silting connected}
\begin{enumerate}[{\rm (1)}]
\item Let $T,U$ be basic silting objects in $\TT$. 
We say that $U$ is \emph{connected} (respectively, \emph{left-connected}) to $T$
if $U$ can be obtained from $T$ by iterated irreducible mutation (respectively, left mutation) on $\silt\TT$. 
\item A triangulated category $\TT$ is called \emph{silting-connected} 
if all basic silting objects in $\TT$ are connected to each other. 
We say that $\TT$ is \emph{strongly silting-connected} 
if for any silting objects $T,U$ of $\TT$ with $T\geq U$, $U$ is left-connected to $T$. 
When any silting object in $\TT$ is tilting, $\TT$ is sometimes called \emph{tilting-connected} 
if it is silting-connected. 
\end{enumerate}
\end{definition}

We pose the following natural question (cf. \cite{AI}).

\begin{question}\label{silting-connected question}
When is a triangulated category $\TT$ silting-connected?
\end{question}

If $A$ is a symmetric algebra and $\TT:=\K^{\rm b}(\proj A)$, 
then Question \ref{tilting-connected question} is nothing but Question \ref{silting-connected question}: 
see Example \ref{symmetric silting=tilting}.

Iyama and the author gave some answers to Question \ref{silting-connected question}.

\begin{example}\cite[Corollary 2.43, Theorem 3.1]{AI}\label{local and hereditary}
Let $A$ be a finite dimensional algebra over a field. 
Then the following hold:
\begin{enumerate}[{\rm (1)}] 
\item If $A$ is local, then $\K^{\rm b}(\proj A)$ is strongly silting-connected; 
\item If $A$ is hereditary, then $\K^{\rm b}(\proj A)$ is silting-connected. 
\end{enumerate}
\end{example}

The following example says that there is a silting-connected triangulated category $\TT$ which is not strong. 

\begin{example}\cite[Example 2.46]{AI}
Let $A$ be the path algebra of the quiver $\xymatrix{1 \ar@<0.2em>[r] \ar@<-0.2em>[r] & 2}$ 
and $\TT:=\K^{\rm b}(\proj A)$. 
Then $\TT$ is silting-connected (Example \ref{local and hereditary}).
We observe that $T:=S_1\oplus P_2[1]$ is a silting object in $\TT$ satisfying $A\geq T$. 
$T$ is left-connected to the irreducible right mutation $\mu^-_{P_1}(A)$ of $A$ with respect to $P_1$, 
and so it is connected to $A$. 
However $T$ is not left-connected to $A$. 
Indeed, the AR-quiver of $\TT$ contains the connected component:
\[\xymatrix@R0.2cm{
\cdots \ar@<0.2em>[rd]\ar@<-0.2em>[rd] &   & P_1 \ar@<0.2em>[rd]\ar@<-0.2em>[rd]  & & X_2 \ar@<0.2em>[rd]\ar@<-0.2em>[rd] & 
& \cdots \\ 
& P_2 \ar@<0.2em>[ru]\ar@<-0.2em>[ru] &     & X_1 \ar@<0.2em>[ru]\ar@<-0.2em>[ru] & & X_3 \ar@<0.2em>[ru]\ar@<-0.2em>[ru] &
}\]
We see that there exists a unique infinite sequence 
\[A=P_1\oplus P_2>P_1\oplus X_1 > X_2\oplus X_1>X_2\oplus X_3>\cdots\]
of iterated irreducible left mutations of $A$ such that each silting object is greater than $T$. 
Thus $T$ is connected to $A$ but not left-connected to $A$.
\end{example}

The theorem below plays an important role in Section \ref{symmetric}. 

\begin{theorem}\label{finite mutation2}
Let $T,U$ be basic silting objects $\TT$ with $T\geq U$. 
If there exist only finitely many $P\in\silt\TT$ satisfying $T\geq P\geq U$, 
then $U$ is left-connected to $T$. 
\end{theorem}
\begin{proof}
By the proof of Theorem \ref{finite mutation1}, 
we observe that $U$ is a direct summand of a silting object $T'$ which is left-connected to $T$. 
Hence it follows from Theorem \ref{partial order} that $U$ is isomorphic to $T'$, 
which implies that $U$ is left-connected to $T$. 
\end{proof}

In the rest of this section, we give an application. 

\begin{definition}
We say that a triangulated category $\TT$ is \emph{silting-discrete} 
if for any silting objects $T,U$ in $\TT$ with $T\geq U$, 
there exist only finitely many $P\in\silt\TT$ satisfying $T\geq P\geq U$. 
\end{definition}

\begin{example}\label{dimension and transitivity}
If $\TT=\add\{M[i]\ |\ i\in\mathbb{Z}\}$ for some object $M$ of $\TT$, 
then $\TT$ is silting-discrete. 
For example let $A$ be a piecewise hereditary algebra of Dynkin type. 
Then putting $M$ equal to a direct sum of non-isomorphic indecomposable $A$-modules, 
we obtain $\K^{\rm b}(\proj A)=\add\{M[i]\ |\ i\in\mathbb{Z}\}$. 
Hence $\K^{\rm b}(\proj A)$ is silting-discrete. 
\end{example}

We have the following observation. 

\begin{proposition}
The following are equivalent:
\begin{enumerate}[{\rm (i)}]
\item $\TT$ is silting-discrete;
\item For any basic silting object $T$ in $\TT$ and any non-negative integer $\ell$, 
there exist only finitely many $P\in\silt\TT$ satisfying $T\geq P\geq T[\ell]$. 
\item $\TT$ has a basic silting object $M$ such that for any non-negative integer $\ell$, 
there exist only finitely many $P\in\silt\TT$ satisfying $M\geq P\geq M[\ell]$. 
\end{enumerate}
\end{proposition}
\begin{proof}
The implications (i)$\Rightarrow$(ii)$\Rightarrow$(iii) are obvious. 
We show the implication (iii)$\Rightarrow$(i). 
Let $T,U$ be basic silting objects in $\TT$ satisfying $T\geq U$. 
Since $M$ is a silting object, applying shifts, we can assume $M\geq T$. 
Since $U$ is a silting object, there is an integer $n$ such that $U\geq M[n]$. 
Thus we have $M\geq T\geq U\geq M[n]$. 
Hence the assertion holds 
since there exist only finitely many $P\in\silt\TT$ satisfying $M\geq P\geq M[n]$. 
\end{proof}

The corollary below is an immediate consequence of Theorem \ref{finite mutation2}. 

\begin{corollary}\label{finite transitivity}
If $\TT$ is silting-discrete, then it is strongly silting-connected. 
\end{corollary}


\section{Silting objects induced by torsion classes}\label{torsion}

In this section we show that any covariantly finite torsion class gives rise to a silting object.
The last result of this section plays an important role in proving our main result. 

Throughout this section let $A$ be a finite dimensional algebra over a field and $\TT:=\K^{\rm b}(\proj{A})$. 

For any full subcategory $\EEE$ of $\mod{A}$, we define full subcategories of $\mod{A}$ by 
\[{^{\perp}}{\EEE}=\{M\in \mod{A}\ |\ \homo{A}{M}{E}=0\ \mbox{for any } E\in \EEE \};\]
\[{\EEE}^{\perp}=\{M\in \mod{A}\ |\ \homo{A}{E}{M}=0\ \mbox{for any } E\in \EEE \}.\]
For any $A$-module $X$, we simply write ${}^\perp X$ (respectively, $X^\perp$) 
instead of ${}^{\perp}(\add{X})$ (respectively, $(\add{X})^\perp$). 

We recall the definition of a torsion class. 

\begin{definitiontheorem}\cite{ASS}
Let $\CCC$ be a full subcategory of $\mod A$. 
We say that $\CCC$ is a \emph{torsion class} in $\mod A$
if $\CCC$ is closed under extensions and taking factor modules. 

Let $\CCC$ be a torsion class in $\mod A$. 
For any $A$-module $X$, 
there exists an exact sequence $0\to X'\to X\to X''\to0$ with $X'\in\CCC$ and $X''\in\CCC^\perp$. 
We call $X'$ a \emph{torsion part} of $X$ and denote it by $\torsion(X)$. 
Moreover the exact sequence above implies that there is an equality $\CCC={}^\perp(\CCC^\perp)$. 
%
%
\end{definitiontheorem}

There is a fundamental example of a torsion class in $\mod{A}$. 

\begin{example}\label{lemma:torsion_ex}\cite[Example VI.1.2(a)]{ASS}
For any $A$-module $X$, the full subcategory ${}^\perp X$ of $\mod A$ is a torsion class. 
\end{example}

\begin{definition}
Let $\EEE$ be a full subcategory of $\mod{A}$ which is closed under extensions and $X$ be an $A$-module in $\EEE$. 
We say that $X$ is \emph{Ext-projective} (respectively, \emph{Ext-injective}) in $\EEE$ 
if $\ext{A}{1}{X}{E}=0$ (respectively, $\ext{A}{1}{E}{X}=0$) for any $A$-module $E$ in $\EEE$.   
\end{definition}

We denote by $\tau$ the Auslander-Reiten translation of $A$.

The following result gives a characterization of Ext-projective and Ext-injective modules. 

\begin{lemma}\label{lemma:torsion_Ext} \cite[Lemma 2.6, 2.12]{AH}
Let $\CCC$ be a torsion class in $\mod A$ and $X$ be in $\CCC$. 
\begin{enumerate}[{\rm (i)}]
\item $X$ is Ext-injective in $\CCC$ if and only if it is isomorphic to $\torsion(I)$ for some $I\in\add DA$; 
\item $X$ is Ext-projective in $\CCC$ if and only if $\tau X$ belongs to $\CCC^\perp$.
\end{enumerate}
\end{lemma}

For a full subcategory $\EEE$ of $\mod{A}$, we define the annihilator of $\EEE$ to be the ideal 
\[\ann{\EEE}=\{a\in A\ |\ Ea=0\ \mbox{for any } E\in \EEE \}.\]

Let $\CCC$ be a torsion class in $\mod{A}$ and let $B$ be the factor algebra $A/\ann{\CCC}$ of $A$ by $\ann{\CCC}$. 
Then it is easy to see that any $A$-module in $\CCC$ can be naturally regarded as a right $B$-module, 
and $\torsion(DA)$ is isomorphic to $DB$ as a right $B$-module (cf. \cite{S}). 

It is well-known that a tilting module induces a torsion class. 
Moreover such a torsion class is characterized by covariantly finiteness.  

\begin{lemma}\label{lemma:tilt_torsion} \cite{S}
Let $\CCC$ be a torsion class in $\mod A$ 
and $X$ be a direct sum of non-isomorphic indecomposable Ext-projective modules in $\CCC$. 
Then the following are equivalent:
\begin{enumerate}[$(1)$]
\item $X$ is a tilting module as a right $A/\ann{\CCC}$-module; 
\item $\CCC$ is covariantly finite in $\mod A$.
\end{enumerate}
\end{lemma}

\begin{remark}\label{remark:contrav_finite_finite}
For any $A$-module $M$, $\add{M}$ is covariantly finite in $\mod{A}$. 
Hence if $A$ is representation-finite, 
then the condition (2) in Lemma \ref{lemma:tilt_torsion} automatically holds. 
\end{remark}

Now we construct a complex in $\K^{\rm b}(\proj{A})$ induced by a torsion class (cf. \cite{AH}). 

\begin{definition}\label{asilting}
Let $\CCC$ be a torsion class in $\mod A$. 
Let $X$ be a direct sum of non-isomorphic indecomposable Ext-projective modules in $\CCC$, 
and $V$ a direct sum of non-isomorphic indecomposable injective modules in $\CCC^\perp$. 
We define a complex $T$ in $\K^{\rm{b}}(\proj A)$ as 
\[T:=T_\CCC:=
\begin{cases}
\begin{array}{c}
\xymatrix@R=0.5pt{
{\rm (0th)} & {\rm (1st)} \\
P_1 \ar[r]^{p} & P_0 \\
\oplus & \\
\nu^{-1}V &
}
\end{array}
\end{cases}\]
where $P_1 \xrightarrow{p}P_0$ is a projective presentation of $X$.
%
\end{definition}

\begin{example}\label{Okuyama-Rickard}
Let $e$ be an idempotent of $A$ and $I:=\nu(1-e)A$.
We see that $T_{{}^\perp I}$ is isomorphic to the complex:
\[T:=
\begin{cases}
\begin{array}{c}
\xymatrix@R=0.5pt{
{\rm (0th)} & {\rm (1st)} \\
P(eA(1-e)A) \ar[r]^(0.7){p_e} & eA \\
\oplus & \\
(1-e)A &
}
\end{array}
\end{cases}
\]
where $p_e$ gives a projective cover of the submodule $eA(1-e)A$ of $eA$. 
Actually, we observe that any injective $A$-module in $({}^\perp I)^\perp$ belongs to $\add I$ 
and see that $M:=eA/eA(1-e)A$ is an Ext-projective module in ${}^\perp I$. 
We shall show that any Ext-projective module in ${}^\perp I$ belongs to $\add M$. 
Note that all composition factors of any $A$-module in ${}^\perp I$ belong to $\add(\operatorname{top}eA)$. 
Let $X$ be in ${}^\perp I$. 
We take a minimal right $\add M$-approximation $f:M'\to X$. 
Since $M$ is an Ext-projective module in ${}^\perp I$, 
we obtain $\homo{A}{M}{\operatorname{coker}f}=0$.
This implies $\operatorname{coker}f=0$, so $f$ is an epimorphism. 
Thus we have an exact sequence $0\to Y\to M'\xrightarrow{f} X\to0$ with $Y\in {}^\perp I$ and $M'\in\add M$. 
In particular if $X$ is an Ext-projective module in ${}^\perp I$, then it belongs to $\add M$. 
\end{example}

As the first step toward the main result of this section, we show the lemma below.  

\begin{lemma}\label{vanishing}
Let $\CCC$ be a torsion class in $\mod A$. 
Then $T_{\CCC}$ is a partial silting object in $\TT$. 
\end{lemma}
\begin{proof}
We use the notation of Definition \ref{asilting}. 
By Lemma \ref{lemma:torsion_Ext}, we have that $H^0(\nu T)\simeq \tau X\oplus V$ belongs to $\CCC^\perp$. 
Hence by Lemma \ref{Serre} we obtain isomorphisms
\[\def\arraystretch{1.5}
\begin{array}{rl}
\homo{\TT}{T}{T[1]} &\simeq D\homo{\K^{\rm b}(\mod{A})}{T[1]}{\nu T} \\
                                      &\simeq D\homo{A}{H^1(T)}{H^0(\nu T)} \\
                                      &=0,
\end{array}\] 
which implies $\homo{\K^{\rm b}(\proj A)}{T}{T[i]}=0$ for any $i>0$. 
Since $A[-1]\geq T\geq A$, 
it follows from Proposition \ref{Bongartz-type} that $T$ is a partial silting object. 
\end{proof}

We now state the main result of this section. 

\begin{theorem}\label{silting induced by torsion pair}
Let $\CCC$ be a torsion class in $\mod A$. 
If $\CCC$ is covariantly finite in $\mod A$, 
then $T:=T_\CCC$ is a silting object in $\TT$. 
Moreover if $\nu\CCC$ is contained in $\CCC$, then $T$ is a tilting object. 
\end{theorem}
\begin{proof}
Let $X$ and $V$ be as in Definition \ref{asilting}. 
For an object $M$ of $\mod A$ or $\TT$, 
we denote by $|M|$ the number of non-isomorphic indecomposable direct summands of $M$. 
Since $\CCC$ is covariantly finite in $\mod{A}$, $X$ is a tilting $A/\ann{\CCC}$-module by Lemma \ref{lemma:tilt_torsion}. 
Hence we obtain equalities 
$|X|=|A/\ann\CCC|=|\torsion(DA)|=|DA/V|$, 
which imply $|T|=|X|+|V|=|DA|=|A|$. 
Since $T$ is a partial silting object by Lemma \ref{vanishing},
it follows from Proposition \ref{silting indecomposable number} that $T$ has to be a silting object. 

Assume $\CCC\supseteq\nu\CCC$. 
Then we also have $\CCC^\perp\supseteq\nu^{-1}(\CCC^\perp)$. 
Since $\tau X$ is in $\CCC^\perp$ by Lemma \ref{lemma:torsion_Ext}, 
we observe that $H^0(T)\simeq\nu^{-1}\tau X\oplus \nu^{-1}V$ belongs to $\CCC^\perp$. 
Hence we get an isomorphism 
\[\homo{\TT}{T}{T[-1]}\simeq \homo{A}{H^1(T)}{H^0(T)}=0.\]
Thus the last assertion holds.
\end{proof}


\section{Tilting-connectedness for representation-finite symmetric algebras}\label{symmetric}

In this section we consider silting-connectedness of $\K^{\rm b}(\proj A)$. 
In particular, we prove the main result of this paper (Theorem \ref{symm finite case}).
The assumption that $A$ is not only self-injective, but is symmetric, finally plays a role. 

Throughout this section let $A$ be a finite dimensional algebra over an algebraically closed field and $\TT:=\K^{\rm b}(\proj{A})$. 

The proposition below has the key to prove our main result. 

\begin{proposition}\label{A}
$\TT$ is silting-connected 
if, for any algebra $\Lambda$ which is derived equivalent to $A$, 
the following conditions are satisfied:
\begin{enumerate}
\item[{\rm (A1)}] Let $T$ be a basic silting object in $\K^{\rm b}(\proj\Lambda)$ with $\Lambda[-1]\geq T\geq \Lambda$. 
Then $T$ is connected to both $\Lambda$ and $\Lambda[-1]$; 
\item[{\rm (A2)}] Let $P$ be a basic tilting object in $\K^{\rm b}(\proj\Lambda)$ with $\Lambda[-\ell]\geq P\geq \Lambda$ 
for a positive integer $\ell$. 
Then there exists a basic tilting object $T$ in $\K^{\rm b}(\proj\Lambda)$ satisfying $\Lambda[-1]\geq T\geq \Lambda$ 
such that $T[-\ell+1]\geq P\geq T$; 
\item[{\rm (A3)}] For any basic silting object $T$ in $\K^{\rm b}(\proj\Lambda)$, 
there exists a basic tilting object which is connected to $T$. 
\end{enumerate}
\end{proposition}
\begin{proof}
By (A3), we shall show that any basic tilting object is connected to $A$. 
Let $P$ be a basic tilting object in $\K^{\rm b}(\proj A)$. 
Note that all shifts of $A$ are connected to $A$ by (A1). 
Therefore we can assume $A[-\ell]\geq P\geq A$ for some $\ell>0$. 
Let $T$ be a basic tilting object of $\K^{\rm b}(\proj A)$ as in (A2). 
Since $A[-1]\geq T\geq A$, 
we observe that $T$ is connected to $A$ by (A1). 
Since $T$ is a tilting object, 
we obtain an equivalence $F:\K^{\rm b}(\proj A)\xrightarrow{\sim} \K^{\rm b}(\proj\ehomo{\TT}{T})$ of triangulated categories 
which sends $T$ to $\ehomo{\TT}{T}$. 
Applying (A2) to $\ehomo{\TT}{T}$, 
we have a basic tilting object $T'$ in $\K^{\rm b}(\proj \ehomo{\TT}{T})$ 
satisfying $FT[-\ell+1]\geq T'[-\ell+2]\geq FP\geq T'\geq FT$. 
Hence we get a sequence 
\[A[-\ell]\geq T_1[-\ell+1]\geq T_2[-\ell+2]\geq P\geq T_2\geq T_1\geq A\]
of basic tilting objects in $\K^{\rm b}(\proj A)$, where $T_1:=T$ and $T_2:=F^{-1}T'$. 
As $T_1[-1]\geq T_2\geq T_1$, 
we obtain that $T_2$ is connected to $T_1$ by (A1). 
This implies that $T_2$ is connected to $A$. 
Continuing the argument above, 
we see that $P$ is connected to $A$. 
\end{proof}


The aim of this section is to prove the following result. 

\begin{theorem}\label{symm finite case}
$\K^{\rm b}(\proj A)$ is tilting-connected 
if $A$ is a representation-finite symmetric algebra. 
\end{theorem}




To prove this theorem, we check that the three conditions in Proposition \ref{A} are satisfied. 

Let us start with an observation on satisfying the condition (A1).

\begin{lemma}\label{A1}
If $A$ is representation-finite, 
then there exist only finitely many $T\in\silt\TT$ satisfying $A[-1]\geq T\geq A$. 
Hence the condition {\rm{(A1)}} holds.
\end{lemma}
\begin{proof}
Note first that every complex of length $2$ in $\TT$ is 
a shift of a projective presentation of an $A$-module up to projective direct summand. 
Since $A$ is representation-finite, 
there exist only finitely many $T\in\silt\TT$ satisfying $A[-1]\geq T\geq A$ 
by Proposition \ref{prop:partial order and length}. 
Hence it follows from Theorem \ref{finite mutation2} that $T$ is left-connected to $A[-1]$ and $A$ is left-connected to $T$. 
Thus the condition (A1) holds. 
\end{proof}

Next we show that a good algebra meets the condition (A2). 

\begin{lemma}\label{A2}
The condition {\rm{(A2)}} holds if $A$ is a representation-finite self-injective algebra. 
\end{lemma}

To prove this lemma, we need the following important result. 

\begin{lemma}\label{lemma:main}
Let $P$ be a presilting object in $\TT$ 
with $A[-\ell]\geq P\geq A$ for a positive integer $\ell$. 
If ${}^\perp H^0(\nu P)$ is covariantly finite in $\mod A$, 
then there exists a silting object $T$ with $A[-1]\geq T\geq A$ 
satisfying $T[-\ell+1]\geq P\geq T$. 
Moreover if $A$ is self-injective and $\add P=\add\nu P$, then $T$ is a tilting object.
\end{lemma}
\begin{proof}
Note that $P$ is isomorphic to a complex $(\cdots\to0\to P^0\to\cdots\to P^\ell\to0\to\cdots)$ 
by Proposition \ref{prop:partial order and length}. 
Put $X:=H^0(\nu P)$ and $\CCC:={}^\perp X$. 
By Example \ref{lemma:torsion_ex}, the full subcategory $\CCC$ is a torsion class in $\mod A$. 
Since $\CCC$ is covariantly finite in $\mod{A}$, 
$T:=T_\CCC$ is a silting object by Theorem \ref{silting induced by torsion pair}. 

(i) We show $P\geq T$. 
It is enough to prove $\homo{\TT}{P}{T[1]}=0$. 
Since $H^1(T)$ belongs to $\CCC$, 
by Lemma \ref{Serre} we obtain isomorphisms
\[\def\arraystretch{1.5}
\begin{array}{rl}
\homo{\TT}{P}{T[1]} &\simeq D\homo{\K^{\rm b}(\mod A)}{T[1]}{\nu P} \\
                    &\simeq D\homo{A}{H^1(T)}{X} \\
                    &=0.
\end{array}\]

(ii) We show $T[-\ell+1]\geq P$. 
We only have to prove $\homo{\TT}{T}{P[\ell]}=0$. 
Since by Lemma \ref{Serre} we obtain isomorphisms
\[\homo{A}{H^{\ell}(P)}{X}\simeq \homo{\K^{\rm b}(\mod A)}{P[\ell]}{\nu P}\simeq D\homo{\TT}{P}{P[\ell]}=0,\]
we see that $H^\ell (P)$ belongs to $\CCC$. 
By Lemma \ref{lemma:torsion_Ext}, we observe that $H^0(\nu T)$ is in $\CCC^\perp$. 
By Lemma \ref{Serre}, we get isomorphisms
\[\def\arraystretch{1.5}
\begin{array}{rl}
\homo{\TT}{T}{P[\ell]} &\simeq D\homo{\K^{\rm b}(\mod A)}{P[\ell]}{\nu T} \\
                       &\simeq D\homo{A}{H^\ell (P)}{H^0(\nu T)} \\
                       &=0.
\end{array}
\] 
Thus the first assertion holds. 

Assume that $A$ is a self-injective algebra and $\add P=\add\nu P$. 
Then $\nu$ and $\nu^{-1}$ commute with $H^0$. 
Since $\add P=\add\nu P$, we have $\add X=\add\nu^{-1}X$. 
This implies $\CCC\supseteq\nu\CCC$. 
Hence the last assertion follows from Theorem \ref{silting induced by torsion pair}.
\end{proof}

Now we are ready to prove Lemma \ref{A2}. 

Let $P$ be a basic tilting object with $A[-\ell]\geq P\geq A$ for some $\ell>0$. 
Take a torsion class $\CCC$ as in the proof of Lemma \ref{lemma:main}. 
By \cite[Theorem 2.1]{AR} (see Theorem \ref{tilting stable}), we have an isomorphism $P\simeq \nu P$. 
Since $\CCC$ is covariantly finite by Remark \ref{remark:contrav_finite_finite}, 
we obtain a tilting object $T$ with $A[-1]\geq T\geq A$ and $T[-\ell+1]\geq P\geq T$ by Lemma \ref{lemma:main}. 
\qed
\medskip

Now we are ready to prove Theorem \ref{symm finite case}. 

We have (A1) by Lemma \ref{A1} and (A2) by Lemma \ref{A2}. 
Since $A$ is a symmetric algebra, any silting object is tilting by Example \ref{symmetric silting=tilting}. 
This immediately implies (A3). 
Hence it follows from Proposition \ref{A} that $\TT$ is tilting-connected. 
%
\qed
\medskip 

We now show the following statement, which is stronger than Theorem \ref{symm finite case}. 

\begin{theorem}
$\K^{\rm b}(\proj A)$ is silting-discrete 
if $A$ is a representation-finite symmetric algebra.
\end{theorem}
\begin{proof}
Note that any silting object in $\TT$ is tilting, since $A$ is symmetric. 
Let $\ell>0$. 
We prove that there exist only finitely many $P\in\silt\TT$ 
satisfying $A[-\ell]\geq P\geq A$ by induction on $\ell$.
If $\ell=1$, then it is evident by Lemma \ref{A1}. 
Assume $\ell>1$. 
It follows from Lemma \ref{A2} 
that for any basic tilting object $P$ in $\TT$ with $A[-\ell]\geq P\geq A$, 
there exists a basic tilting object $T_P$ satisfying $A[-\ell]\geq T_P[-\ell+1]\geq P\geq T_P\geq A$. 
By Lemma \ref{A1}, the set $\{T_P\in\silt\TT\ |\ A[-\ell]\geq P\geq A\}$ is finite. 
Hence by the induction hypothesis, the assertion holds. 
\end{proof}

As a consequence, we obtain a Bongartz-type Lemma for representation-finite symmetric algebras. 

\begin{corollary}\cite[Theorem 3.6]{AH}\label{symm finite Bongartz-type}
If $A$ is a representation-finite symmetric algebra, 
then any pretilting object $T$ in $\TT$ is partial tilting. 
\end{corollary}
\begin{proof}
Since $A$ is symmetric, we clearly have an isomorphism $T\simeq \nu T$. 
So we can apply Lemma \ref{lemma:main}. 
By the same argument as in the proof of Proposition \ref{A}, 
we have a tilting object $U$ satisfying $U[-1]\geq T\geq U$. 
Hence it follows from Proposition \ref{Bongartz-type} that $T$ is a partial tilting object. 
\end{proof}

%

\begin{example}
Let $A$ be the algebra presented by the quiver 
$\xymatrix{
1 \ar[r]_{x_1} & 2 \ar[r]_{x_2} & 3 \ar@/_0.8em/[ll]_{x_3}
}$
with relations $x_1x_2x_3x_1=x_2x_3x_1x_2=x_3x_1x_2x_3=0$.
Then $A$ is a representation-finite symmetric algebra. 
Define an object $T$ of $\TT$ as a complex
\[T:=\bigoplus
\begin{cases}
\begin{array}{c}
\xymatrix@R=0.5pt{
{\rm (0th)} & {\rm (1st)} & {\rm (2nd)} \\
P_2 \ar[r]^{x_2x_3x_1} & P_2 \ar[r]^{x_1} & P_1 \\
P_2 &     &     \\
P_2 \ar[r]^{x_2} & P_3. &
}
\end{array}
\end{cases}\] 
We can easily check that $T$ is a tilting object. 
We observe 
${}^\perp H^0(T)=\add\left(S_1\oplus S_3\oplus \def\arraystretch{0.5}\left(\begin{array}{c}3\\1\end{array}\right)\right)$. 
By Lemma \ref{lemma:torsion_Ext}, it is obvious that  
$S_1$ and $\def\arraystretch{0.5}\left(\begin{array}{c}3\\1\end{array}\right)$ are Ext-projective modules in ${}^\perp H^0(T)$.
Hence we take a tilting object $T_1:=T_\CCC$ as a complex 
\[T_1=\bigoplus
\begin{cases}
\begin{array}{c}
\xymatrix@R=0.5pt{
{\rm (0th)} & {\rm (1st)} \\
P_2 \ar[r]^{x_1} & P_1 \\
P_2 & \\
P_2 \ar[r]^{x_3x_1} & P_3. 
}
\end{array}
\end{cases}\] 
We see that $T_1$ is an irreducible left mutation of $T$. 
Moreover, we have an isomorphism $A\simeq \mu^+_{P_2\to P_1}\mu^+_{P_2\to P_3}(T_1)$. 
Thus $A$ is left-connected to $T$. 
\end{example}

The silting quiver, which we will now define, is very useful for observing the behavior of silting objects 
under iterated irreducible silting mutation. 

\begin{definition}
The \emph{silting quiver} of $\TT$ is defined as follows:
\begin{itemize}
\item The set of vertices is $\silt\TT$;
\item We draw an arrow $T\to U$ if $U$ is an irreducible left mutation of $T$. 
\end{itemize}
To simplify the notation, we identify each basic silting object $T$ with $T[i]$ for all integers $i$ 
and use arrows $\xymatrix{T\ar[r]|(0.45){[1]}&U}$ which mean $\xymatrix{T\ar[r]&U[1]}$. 
\end{definition}

Note that the silting quiver of $\TT$ is connected if and only if $\TT$ is silting-connected. 

\begin{example}\cite[Example 2.48]{AI}
Let $A$ be the algebra presented by the quiver $\xymatrix{1\ar@<0.2em>[r]^{a}&2\ar@<0.2em>[l]^{b}}$ 
with relations $(ab)^\ell a=0=(ba)^\ell b$ ($\ell\geq1$). 
Then $A$ is a representation-finite symmetric algebra. 
By Theorem \ref{symm finite case} the silting quiver of $\TT$ is connected, 
and in fact it is the following quiver: 
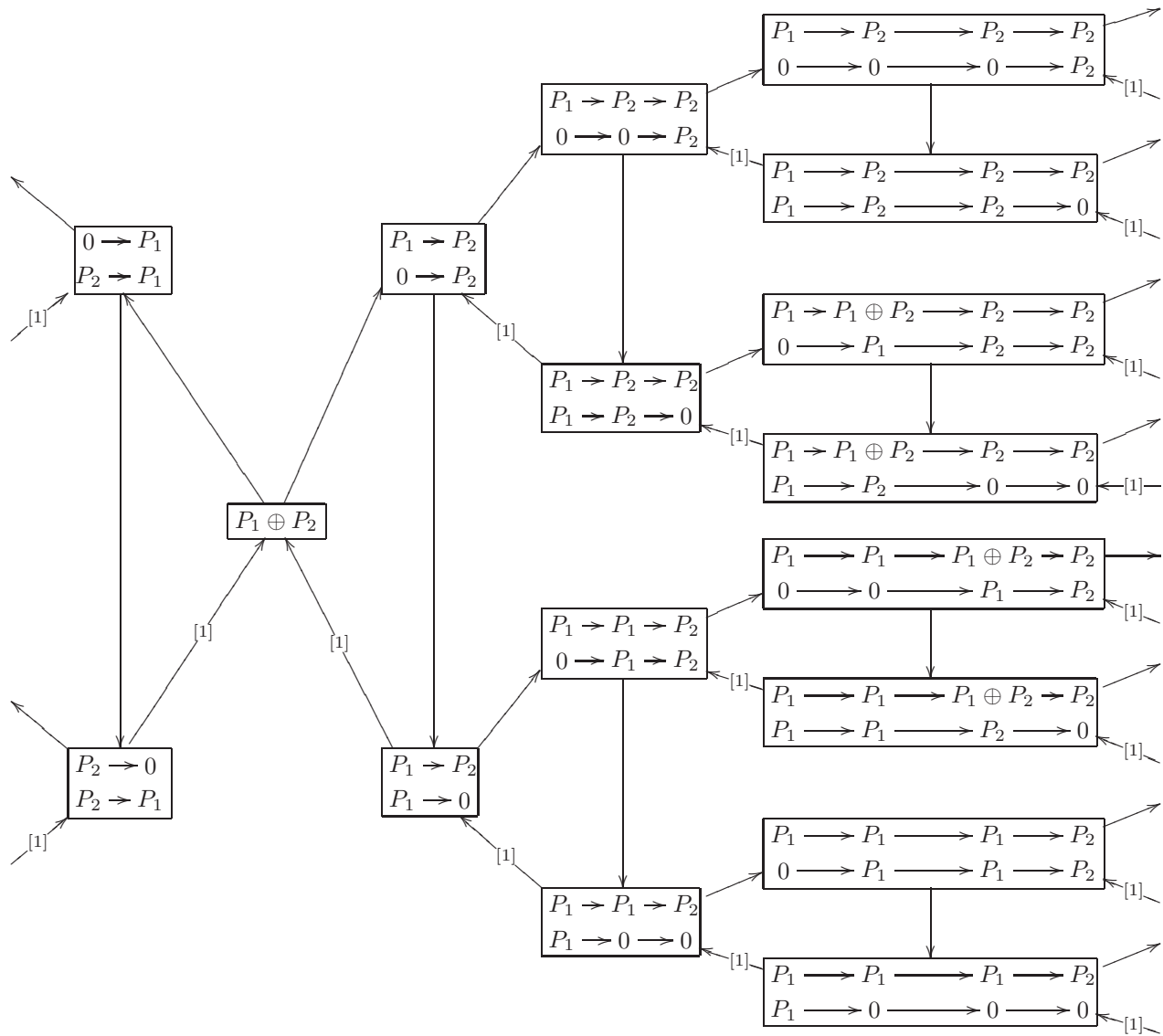
\[
{\small{
\xymatrix@R=0.1pt @C=0.3cm{
&& & && && & && & & && & & & &&\\
&& & && && & && & & && P_1 \ar[r] & P_2 \ar[r] & P_2 \ar[r] & P_2 &&\\
&& & && && & && & & && 0 \ar[r] & 0 \ar[r] & 0 \ar[r] & P_2 &&\\
&& & && && & && P_1 \ar[r] & P_2 \ar[r] & P_2 && & & & &&\\
&& & && && & && 0 \ar[r] & 0 \ar[r] & P_2 && & & & &&\\   
&& & && && & && & & && P_1 \ar[r] & P_2 \ar[r] & P_2 \ar[r] & P_2 &&\\
&& & && && & && & & && P_1 \ar[r] & P_2 \ar[r] & P_2 \ar[r] & 0 &&\\
&& 0 \ar[r] & P_1 && && P_1 \ar[r] & P_2 && & & && & & & &&\\
&& P_2 \ar[r] & P_1 && && 0 \ar[r] & P_2 && & & && & & & &&\\
&& & && && & && & & && P_1 \ar[r] & P_1\oplus P_2 \ar[r] & P_2 \ar[r] & P_2 &&\\
&& & && && & && & & && 0 \ar[r] & P_1 \ar[r] & P_2 \ar[r] & P_2 &&\\
&& & && && & && P_1 \ar[r] & P_2 \ar[r] & P_2 && & & & &&\\
&& & && && & && P_1 \ar[r] & P_2 \ar[r] & 0 && & & & &&\\
&& & && && & && & & && P_1 \ar[r] & P_1\oplus P_2 \ar[r] & P_2 \ar[r] & P_2 &&\\
&& & && && & && & & && P_1 \ar[r] & P_2 \ar[r] & 0 \ar[r] & 0 &&\\
&& & && P_1\oplus P_2 && & && & & && & & & &&\\
&& & && && & && & & && P_1 \ar[r] & P_1 \ar[r] & P_1\oplus P_2 \ar[r] & P_2 &&\\
&& & && && & && & & && 0 \ar[r] & 0 \ar[r] & P_1 \ar[r] & P_2 &&\\
&& & && && & && P_1 \ar[r] & P_1 \ar[r] & P_2 && & & & &&\\
&& & && && & && 0 \ar[r] & P_1 \ar[r] & P_2 && & & & &&\\
&& & && && & && & & && P_1 \ar[r] & P_1 \ar[r] & P_1\oplus P_2 \ar[r] & P_2 &&\\
&& & && && & && & & && P_1 \ar[r] & P_1 \ar[r] & P_2 \ar[r] & 0 &&\\
&& P_2 \ar[r] & 0   && && P_1 \ar[r] & P_2 && & & && & & & &&\\
&& P_2 \ar[r] & P_1 && && P_1 \ar[r] & 0 && & & && & & & &&\\
&& & && && & && & & && P_1 \ar[r] & P_1 \ar[r] & P_1 \ar[r] & P_2 &&\\
&& & && && & && & & && 0 \ar[r] & P_1 \ar[r] & P_1 \ar[r] & P_2 &&\\
&& & && && & && P_1 \ar[r] & P_1 \ar[r] & P_2 && & & & &&\\
&& & && && & && P_1 \ar[r] & 0 \ar[r] & 0 && & & & &&\\
&& & && && & && & & && P_1 \ar[r] & P_1 \ar[r] & P_1 \ar[r] & P_2 &&\\
&& & && && & && & & && P_1 \ar[r] & 0 \ar[r] & 0 \ar[r] & 0 &&\\
&& & && && & && & & && & & & &&
\save "2,15"."3,18"*[F]\frm{}="a" \restore
\save "6,15"."7,18"*[F]\frm{}="b" \restore
\save "10,15"."11,18"*[F]\frm{}="c" \restore
\save "14,15"."15,18"*[F]\frm{}="d" \restore
\save "17,15"."18,18"*[F]\frm{}="e" \restore
\save "21,15"."22,18"*[F]\frm{}="f" \restore
\save "25,15"."26,18"*[F]\frm{}="g" \restore
\save "29,15"."30,18"*[F]\frm{}="h" \restore
\save "4,11"."5,13"*[F]\frm{}="i" \restore
\save "12,11"."13,13"*[F]\frm{}="j" \restore
\save "19,11"."20,13"*[F]\frm{}="k" \restore
\save "27,11"."28,13"*[F]\frm{}="l" \restore
\save "8,8"."9,9"*[F]\frm{}="m" \restore
\save "23,8"."24,9"*[F]\frm{}="n" \restore
\save "16,6"*[F]\frm{}="o" \restore
\save "8,3"."9,4"*[F]\frm{}="p" \restore
\save "23,3"."24,4"*[F]\frm{}="q" \restore
\ar@<6em> "a";"b"
\ar@<6em> "c";"d"
\ar@<6em> "e";"f"
\ar@<6em> "g";"h"
\ar@<2.5em> "i";"j"
\ar@<2.5em> "k";"l"
\ar@<1.3em> "m";"n"
\ar@<1.3em> "p";"q"
\ar "2,18";"1,20" 
\ar "6,18";"5,20" 
\ar "10,18";"9,20" 
\ar "14,18";"13,20" 
\ar "17,18";"17,20" 
\ar "21,18";"20,20" 
\ar "25,18";"24,20" 
\ar "29,18";"28,20" 
\ar "4,20";"3,18" |(0.4){[1]} 
\ar "8,20";"7,18" |(0.4){[1]} 
\ar "12,20";"11,18" |(0.4){[1]} 
\ar "15,20";"15,18" |(0.4){[1]} 
\ar "19,20";"18,18" |(0.4){[1]} 
\ar "23,20";"22,18" |(0.4){[1]} 
\ar "27,20";"26,18" |(0.4){[1]} 
\ar "31,20";"30,18" |(0.4){[1]} 
\ar "4,13";{"3,15"+<-0.9em, -0.1em>}
\ar "12,13";{"11,15"+<-0.9em, -0.1em>}
\ar "19,13";{"18,15"+<-0.9em, -0.1em>}
\ar "27,13";{"26,15"+<-0.9em, -0.1em>}
\ar "b";"i" |(0.2){[1]}
\ar "d";"j" |(0.2){[1]}
\ar "f";"k" |(0.2){[1]}
\ar "h";"l" |(0.2){[1]}
\ar "8,9";{"5,11"+<-0.9em, -0.4em>}
\ar "23,9";{"20,11"+<-0.9em, -0.4em>}
\ar "j";"m" |(0.35){[1]}
\ar "l";"n" |(0.35){[1]}
\ar "o";"m"
\ar "n";"o" |{[1]}
\ar "o";"p"
\ar {"q"+<1.7em,0.9em>};"o" |{[1]}
\ar "8,3";"6,1"
\ar "23,3";"21,1"
\ar "11,1";"9,3" |(0.4){[1]}
\ar "26,1";"24,3" |(0.4){[1]}
}}}
\]
\end{example}

\appendix
\def\thesection{\Alph{section}}
\section{Derived invariances of self-injectivity}\label{selfinj}

In the representation theory of algebras, it is important to know which properties of algebras are derived invariant. 
By Lemma \ref{Serre}, it is obvious that the properties of being symmetric and being Iwanaga-Gorenstein are derived invariant. 
However it is non-trivial to show that self-injectivity is derived invariant. 
This was proved by Al-Nofayee and Rickard. 
Here, we show it as an application of the partial order. 

Throughout this section let $A$ be a self-injective algebra over an algebraically closed field $k$ 
and $\TT:=\K^{\rm{b}}(\proj{A})$. 

Since $A$ is self-injective, it is clear that the Nakayama functor $\nu$ of $A$ acts on $\TT$. 
Moreover for any silting (respectively, tilting) object $T$ in $\TT$, 
we see that $\nu T$ is also silting (respectively, tilting). 

\begin{lemma}\label{lemma:tilting closed}
Let $T$ be a basic tilting object in $\TT$. 
Then we have $T\geq \nu T$. 
Moreover any object $X$ of $\TT$ with $T\geq X$ satisfies $T\geq\nu X$. 
\end{lemma}
\begin{proof}
By Lemma \ref{Serre}, we have an isomorphism $\homo{\TT}{T}{\nu T[i]}\simeq D\homo{\TT}{T[i]}{T}=0$ for any $i\neq0$. 
Therefore we obtain $T\geq \nu T$. 
We can easily check that any object $X$ of $\TT$ with $T\geq X$ satisfies $\nu T\geq\nu X$. 
Hence the last assertion follows from Theorem \ref{partial order}. 
\end{proof}


Al-Nofayee and Rickard observed that the following result follows easily from Corollary 9 of \cite{HS}. 
It will play an important role in our discussion. 

\begin{proposition}\label{prop:AR} \cite[Theorem 1.1]{AR}
Let $\Lambda$ be a finite dimensional algebra for a field, 
and $\cdots, P^{-1}, P^{0}, P^{1}, \cdots$ a sequence of finitely generated projective $\Lambda$-modules, 
such that $P^{i}=0$ for all but finitely many $i$. 
Then up to isomorphism there are only finitely many complexes 
\[P=\cdots \to P^{-1}\to P^{0}\to P^{1}\to \cdots\]
satisfying $\homo{\K^{\rm b}(\proj\Lambda)}{P}{P[1]}=0$. 
\end{proposition}

Using this proposition, we deduce the following fact. 

\begin{lemma}\label{lemma:silting power stable}
For any basic silting object $T$ in $\TT$, 
there exists a positive integer $n$ such that $\nu^nT$ is isomorphic to $T$. 
\end{lemma}
\begin{proof}
Applying shifts to $T$, we can assume $A\geq T$. 
By Proposition \ref{resolution}, there exist $\ell\geq 0$ and $P_{0}, P_{1}, \cdots, P_{\ell}\in \proj{A}$ 
with triangles 
\[\xymatrix@R=0.2cm{
T_1 \ar[r] & P_0 \ar[r] & T \ar[r] & T_1[1], \\
\cdots, \\
T_\ell \ar[r] & P_{\ell-1} \ar[r] & T_{\ell-1} \ar[r] & T_\ell[1], \\
0 \ar[r] & P_\ell \ar[r] & T_\ell \ar[r] & 0.
}\] 
Since $A$ is a self-injective algebra, 
we have a positive integer $s$ such that $P_{i}\simeq \nu^{s}P_{i}$ for any $0\leq i\leq \ell$. 
Therefore we also obtain triangles 
\[\xymatrix@R=0.2cm{
\nu^s T_1 \ar[r] & P_0 \ar[r] & \nu^s T \ar[r] & \nu^s T_1[1], \\
\cdots, \\
\nu^s T_\ell \ar[r] & P_{\ell-1} \ar[r] & \nu^s T_{\ell-1} \ar[r] & \nu^s T_\ell[1], \\
0 \ar[r] & P_\ell \ar[r] & \nu^s T_\ell \ar[r] & 0.
}\]
This implies that for each $i$, the $i$-th term of $\nu^s T$ is isomorphic to that of $T$. 
By Proposition \ref{prop:AR}, there exists a multiple $n$ of $s$ such that $T\simeq \nu^{n}T$. 
\end{proof}

Now we recover a result of \cite{AR} which is important in the proof of derived invariance of self-injectivity. 

\begin{theorem}\label{tilting stable}\cite[Theorem 2.1]{AR}
Let $T$ be a basic silting object in $\TT$. 
Then the following are equivalent:
\begin{enumerate}[$(1)$]
\item $T$ is a tilting object;
\item $\nu T$ is isomorphic to $T$.
\end{enumerate}
\end{theorem}
\begin{proof}
We show the implication (2)$\Rightarrow$(1). 
As $T\simeq \nu T$, by Lemma \ref{Serre} we have isomorphisms 
\[\homo{\TT}{T}{T[i]}\simeq D\homo{\TT}{T[i]}{\nu T}\simeq D\homo{\TT}{T}{T[-i]}.\]
This immediately implies that $\homo{\TT}{T}{T[i]}=0$ for any $i<0$. 

We show the implication (1)$\Rightarrow$(2). 
Since $\nu^{j}T$ is tilting for any $j\in \mathbb{Z}$, 
by Lemma \ref{lemma:tilting closed} we have a sequence  
\[T\geq \nu T\geq \nu^{2}T \geq \cdots \]
of basic tilting objects. 
By Lemma \ref{lemma:silting power stable}, we obtain $T\simeq \nu^{n}T$ for some $n>0$. 
Hence the assertion follows from Theorem \ref{partial order}. 
\end{proof}

\begin{corollary}\label{selfinj stable}\cite[Theorem 2.1]{AR}
Self-injectivity of algebras is derived invariance. 
\end{corollary}
\begin{proof}
Let $T$ be a basic tilting object in $\TT=\K^{\rm b}(\proj A)$ and put $B:=\ehomo{\TT}{T}$. 

We will show that there exists an isomorphism $DB\simeq \homo{\TT}{T}{\nu T}$ of right $B$-modules.
By Lemma \ref{Serre}, we have a functorial isomorphism $S:D\homo{\TT}{T}{-}\xrightarrow{\sim}\homo{\TT}{-}{\nu T}$.
This implies that $S_T:DB\to\homo{\TT}{T}{\nu T}$ is an isomorphism of $k$-vector spaces. 
Since $S$ is functorial, 
for any $f:X\to Y$ in $\TT$ we obtain a commutative diagram
\[\xymatrix{
D\homo{\TT}{T}{Y} \ar[r]^{S_Y}_{\sim} \ar[d]_{D\homo{\TT}{T}{f}} & \homo{\TT}{Y}{\nu T} \ar[d]^{\homo{\TT}{f}{\nu T}} \\
D\homo{\TT}{T}{X} \ar[r]_{S_X}^{\sim} & \homo{\TT}{X}{\nu T}.
}\]
Putting $X=Y=T$, by the commutative diagram above 
we find an equality $S_T(\alpha\cdot f)=S_T(\alpha)\cdot f$ for any $\alpha\in DB$. 
Hence we see that $S_T$ is an isomorphism of right $B$-modules. 

On the other hand, by Theorem \ref{tilting stable} we have an isomorphism $\varphi:T\xrightarrow{\sim}\nu T$. 
Therefore $\homo{\TT}{T}{\varphi}:B\to\homo{\TT}{T}{\nu T}$ is an isomorphism as a right $B$-module. 
Thus we obtain $B\simeq DB$ as a right $B$-module, 
which says that $B$ is a self-injective algebra. 
\end{proof}

\begin{remark}
In the proof of Corollary \ref{selfinj stable}, 
we have two algebra isomorphisms 
\begin{align}
\label{through nu}  \nu :B & \xrightarrow{\sim}\ehomo{\TT}{\nu T} \\
\label{through phi}      B & \xrightarrow{\sim}\ehomo{\TT}{\nu T}\ (f\mapsto \varphi f\varphi^{-1})
\end{align}
Thus the $k$-vector space $M:=\homo{\TT}{T}{\nu T}$ has two kinds of structure as a left $B$-module. 

By a similar argument in the proof of Corollary \ref{selfinj stable}, 
regarding $M$ as a left $B$-module through the algebra isomorphism (\ref{through nu}), 
we see that $S_T$ is an isomorphism as a $(B, B)$-bimodule. 
On the other hand, regarding $M$ as a left $B$-module through the algebra isomorphism (\ref{through phi}), 
we observe that $\homo{\TT}{T}{\varphi}$ is an isomorphism as a $(B, B)$-bimodule. 
However $DB$ is not isomorphic to $B$ as a $(B, B)$-bimodule unless $A$ is a symmetric algebra, 
since the algebra isomorphisms (\ref{through nu}) and (\ref{through phi}) are different.  
\end{remark}

We close this paper by giving an application of Lemma \ref{no common summand}. 

Let $\{P_1,\cdots,P_n\}$ be the set of non-isomorphic projective indecomposable $A$-modules. 
Then there exists a permutation $\rho$ of the set $I:=\{1,\cdots,n\}$, called the \emph{Nakayama permutation} of $A$, 
such that $\nu P_i\simeq P_{\rho(i)}$ for any integer $i$ in $I$.

We denote by $\tilt\TT$ the isomorphism classes of basic tilting objects in $\TT$. 

We have the following result, which was proved by Abe and Hoshino. 

\begin{corollary}\label{unique tilting}\cite[Proposition 2.14]{AH}
If the Nakayama permutation of $A$ is transitive, 
then we have $\tilt\TT=\{A[i]\ |\ i\in\mathbb{Z}\}$.
\end{corollary}
\begin{proof}
Let $T$ be a basic tilting object in $\TT$. 
We can assume $A\geq T$ and $H^0(T)\not=0$. 
By Proposition \ref{resolution}, we have triangles 
\[\xymatrix@R=0.2cm{
T_1 \ar[r] & P_0 \ar[r] & T \ar[r] & T_1[1] \\
\cdots, \\
T_\ell \ar[r] & P_{\ell-1} \ar[r] & T_{\ell-1} \ar[r] & T_\ell[1] \\
0 \ar[r] & P_\ell \ar[r] & T_\ell \ar[r] & 0
}\] 
for some $P_0\not=0,P_1,\cdots,P_\ell\in \add A$. 
By Theorem \ref{tilting stable}, we obtain an isomorphism $T\simeq \nu T$, 
and so we get an isomorphism $P_{i}\simeq \nu P_{i}$ for any $0\leq i\leq \ell$. 
Since $P_0\simeq \nu P_0$ and the Nakayama permutation of $A$ is transitive, 
we see that $A$ is in $\add P_0$. 
By Lemma \ref{no common summand}, we have $\ell=0$. 
Hence $T$ belongs to $\add A$, which implies that there exists an isomorphism $T\simeq A$. 
\end{proof}



\end{document}